\documentclass[12pt]{amsart}
\usepackage{amsfonts, amssymb, latexsym, epsfig} 

\usepackage{amscd,amssymb}
\usepackage{verbatim}
\usepackage[mathscr]{euscript}
\usepackage{amsmath}
\input xy
\xyoption{all}

\newsymbol\pp 1275

\usepackage{tikz}
\usetikzlibrary{matrix,arrows,backgrounds,shapes.misc,shapes.geometric,patterns,calc,positioning}

\usepackage[colorlinks=true, pdfstartview=FitV, linkcolor=blue, citecolor=blue, urlcolor=blue]{hyperref}

\usepackage{fullpage}
\usepackage{graphicx}
\usepackage{enumerate}
\def\VR{\kern-\arraycolsep\strut\vrule &\kern-\arraycolsep}
\def\vr{\kern-\arraycolsep & \kern-\arraycolsep}

\newcommand{\im}{\ensuremath{\text{im }}}
 
\newcommand{\dsp}{\ensuremath{\displaystyle}}

\newcommand{\dv}{\ensuremath{\underline{\dim }\, }}

\newcommand{\be}{\begin{enumerate}}

\newcommand{\mbb}[1]{\ensuremath{\mathbb{#1}}}

\newcommand{\snot}{\ensuremath{\; \, \diagdown \; \,}}

\newcommand{\glb}{\text{GL}( \ensuremath{\beta} )}

\newcommand{\la}{\ensuremath{\langle}}
\newcommand{\ra}{\ensuremath{\rangle}}

\newcommand{\req}{\ensuremath{ \mathcal{R}(Q)}}

\newcommand{\vci}[1]{\begingroup
\setbox0=\hbox{#1}%
\parbox{\wd0}{\box0}\endgroup}

\newtheorem{theorem}{Theorem}

\newtheorem{lemma}[theorem]{Lemma}
\newtheorem{prop}[theorem]{Proposition}
\newtheorem{corollary}[theorem]{Corollary}

\theoremstyle{definition}
\newtheorem{definition}[theorem]{Definition}

\newtheorem*{recall}{Recall}

\newtheorem{rmk}{Remark}
\newenvironment{remark}[1][]{\begin{rmk}[#1]\pushQED{\qed}}{\popQED \end{rmk}}

\newtheorem{obs}{Observation}

\newtheorem{ex}{Example}
\newenvironment{example}[1][]{\begin{ex}[#1]\pushQED{\qed}}{\popQED \end{ex}}

\newcommand{\Hom}{\operatorname{Hom}}
\newcommand{\End}{\operatorname{End}}
\newcommand{\Ext}{\operatorname{Ext}}

\newcommand{\rep}{\operatorname{rep}}

\newcommand{\SI}{\operatorname{SI}}
\newcommand{\SL}{\operatorname{SL}}
\newcommand{\GL}{\operatorname{GL}}

\newcommand{\ZZ}{\mathbb Z}

\newcommand{\Mat}{\operatorname{Mat}}

\newcommand{\Stab}{\operatorname{Stab}}

\newcommand{\ddim}{\operatorname{\mathbf{dim}}}
\newcommand{\dd}{\operatorname{\mathbf{d}}}
 \newcommand{\ee}{\operatorname{\mathbf{e}}}

\newcommand{\rSoc}{\text{rSoc}}
\newcommand{\rTop}{\text{rTop}}

\newcommand{\module}{\operatorname{mod}}

\begin{document}
\title{ On locally semi-simple representations of quivers}

\author{Calin Chindris}
\address{University of Missouri-Columbia, Mathematics Department, Columbia, MO, USA}
\email[Calin Chindris]{chindrisc@missouri.edu}

\author{Dan Kline}
\address{University of Missouri-Columbia, Mathematics Department, Columbia, MO, USA}
\email[Dan Kline]{dbkfz9@mail.missouri.edu}

\date{\today}
\bibliographystyle{plain}

\begin{abstract} In this paper, we solve a problem raised by V. Kac in \cite{Kac} on locally semi-simple quiver representations. Specifically, we show that an acyclic quiver $Q$ is of tame representation type if and only if every representation of $Q$ with a semi-simple ring of endomorphisms is locally semi-simple.
\end{abstract}

\maketitle
\setcounter{tocdepth}{1}
\tableofcontents

\section{Introduction}
Throughout, $K$ denotes an algebraically closed field of characteristic zero. All quivers are assumed to be finite, connected, and without oriented cycles. All representations and modules are assumed to be finite-dimensional. By a module, we always mean a left module.

In \cite{Kac}, Kac asked for a representation-theoretic description of the so-called locally semi-simple representations of a quiver $Q$. These representations arise most naturally when studying quiver representations within the general framework of invariant theory. In this context, a $\beta$-dimensional representation is said to be locally semi-simple if its $\SL(\beta)$-orbit is closed in the representation space of $\beta$-dimensional representations. It follows from general results of King \cite{K} and Shmelkin \cite{Shm1} that if $V$ is a locally semi-simple representation of $Q$ then the ring $\End_Q(V)$ of endomorphisms of $V$ is semi-simple (see Section \ref{locally-semi-simple-sec}). In \cite[page 161]{Kac}, Kac suggests that the converse ought to be true when $Q$ is a tame quiver. In this paper, we prove that this is indeed the case. On the other hand, we show that for any wild quiver there exist representations with semi-simple rings of endomorphisms which are not locally semi-simple. Our main result is:

\begin{theorem} \label{main-thm} Let $Q$ be a quiver. Then the following statements are equivalent:
\begin{enumerate}
\item $Q$ is tame;
\item every representation $V$ of $Q$ with a semi-simple ring of endomorphisms is locally semi-simple.
\end{enumerate}
\end{theorem}

The layout of this paper is as follows. In Section \ref{background-sec}, we recall general results from quiver invariant theory and explain how to reformulate Kac's question in terms of orthogonal Schur sequences and stability weights for quivers (see Theorems \ref{locally-semi-simple-gen-theorem} and \ref{ssalg}). In Section \ref{tame-sec}, our goal is to provide a constructive solution to the problem of finding stability weights for orthogonal Schur sequences of representations of tame quivers. This essentially proves one implication of our main result. We complete the proof of Theorem \ref{main-thm} in Section \ref{proof-thm}.\\

\noindent
\textbf{Acknowledgment:} The authors would like to thank Dmitri Shmelkin for clarifying discussions on the results from \cite{Shm1, Shm2}. The authors were supported by the NSA under grant H98230-15-1-0022.

\section{Background} \label{background-sec}
Let $Q=(Q_0,Q_1,t,h)$ be a finite quiver with vertex set $Q_0$ and arrow set $Q_1$. The two functions $t,h:Q_1 \to Q_0$ assign to each arrow $a \in Q_1$ its tail \emph{ta} and head \emph{ha}, respectively.

A representation $V$ of $Q$ over $K$ is a collection $(V(x),V(a))_{x\in Q_0, a\in Q_1}$ of finite-dimensional $K$-vector spaces $V(x)$, $x \in Q_0$, and $K$-linear maps $V(a): V(ta) \to V(ha)$, $a \in Q_1$. The dimension vector of a representation $V$ of $Q$ is the function $\ddim V \colon Q_0 \to \ZZ$ defined by $(\ddim V)(x)=\dim_{K} V(x)$ for $x\in Q_0$. The one-dimensional representation of $Q$ supported at vertex $x \in Q_0$ is denoted by $S_x$ and its dimension vector is denoted by $\ee_x$. By a dimension vector of $Q$, we simply mean a vector $\dd \in \ZZ_{\geq 0}^{Q_0}$.

Let $V$ and $W$ be two representations of $Q$. A morphism $\varphi:V \rightarrow W$ is defined to be a collection $(\varphi(x))_{x \in Q_0}$ of $K$-linear maps with $\varphi(x) \in \Hom_K(V(x), W(x))$ for each $x \in Q_0$, such that $\varphi(ha)V(a)=W(a)\varphi(ta)$ for each $a \in Q_1$. We denote by $\Hom_Q(V,W)$ the $K$-vector space of all morphisms from $V$ to $W$. We say that $V$ is a subrepresentation of $W$ if $V(x)$ is a subspace of $W(x)$ for each $x \in Q_0$ and $(i_x: V(x) \hookrightarrow W(x))_{x \in Q_0}$ is a morphism of representations, i.e. $V(a)$ is the restriction of $W(a)$ to $V(ta)$ for each $a \in Q_1$. The category of all representations of $Q$ is denoted by $\rep(Q)$. It turns out that $\rep(Q)$ is an abelian category. A representation $V \in \rep(Q)$ is called a \emph{Schur representation} if $\End_Q(V) \simeq K$.

Given a quiver $Q$, its path algebra $KQ$ has a $K$-basis consisting of all paths (including the trivial ones), and multiplication in $KQ$ is given by concatenation of paths. It is easy to see that any $KQ$-module defines a representation of $Q$, and vice-versa. Furthermore, the category $\module(KQ)$ of $KQ$-modules is equivalent to the category $\rep(Q)$. In what follows, we identify $\module(KQ)$ and $\rep(Q)$, and use the same notation for a module and the corresponding representation.

The Euler form of $Q$ is the bilinear form $\langle -,- \rangle: \ZZ^{Q_0} \times \ZZ^{Q_0} \to \ZZ$ defined by
\[
\langle \alpha, \beta \rangle=\sum_{x \in Q_0} \alpha(x)\beta(x)-\sum_{a \in Q_1} \alpha(ta) \beta(ha), \forall \alpha, \beta \in \ZZ^{Q_0}.
\]
The corresponding Tits quadratic form is $q:\ZZ^{Q_0}\to \ZZ$, $q(\alpha)=\langle \alpha, \alpha \rangle, \forall \alpha \in \ZZ^{Q_0}.$

\subsection{Semi-stable quiver representations} 
Let $Q$ be a quiver and $\beta \in \ZZ^{Q_0}_{\geq 0}$ a dimension vector of $Q$. The affine space 
\[\rep(Q,\beta):= \prod_{a \in Q_1} \Mat_{\beta(ha)\times \beta(ta)}(K)\]
is called the representation space of $\beta$-dimensional representations of $Q$. It is acted upon by the base change group \[\GL(\beta):=\prod_{x\in Q_0}\GL(\beta(x),K)\] by simultaneous conjugation, i.e., for $g=(g(x))_{x\in Q_0}\in \GL(\beta)$ and $V=(V(a))_{a \in Q_1} \in \rep(Q,\beta)$,  $g \cdot V$ is defined by \[(g\cdot V)(a)=g(ha)V(a) g(ta)^{-1}, \forall a \in Q_1.\]
Under our assumption that $Q$ has no oriented cycles, the ring of invariants $I(Q,\beta):=K[\rep(Q,\beta)]^{\GL(\beta)}$ is just the base field $K$. 

Let us now consider the commutator subgroup $\SL(\beta)=\prod_{x \in Q_0}\SL(\beta(x),K)$ of $\GL(\beta)$ and its action on $K[\rep(Q,\beta)]$. The resulting ring of semi-invariants $\SI(Q,\beta):=K[\rep(Q,\beta)]^{\SL(\beta)}$ is highly non-trivial. It has a weight space decomposition over the group $X^{\star}(\GL(\beta))$ of rational characters of $\GL(\beta)$:
\[\SI(Q,\beta)=\bigoplus_{\chi \in X^\star(\GL(\beta))}\SI(Q,\beta)_{\chi}.\]
For each character $\chi \in X^{\star}(\GL(\beta))$, 
\[\SI(Q,\beta)_{\chi}=\lbrace f \in K[\rep(Q,\beta)] \mid g f= \chi(g)f \text{~for all~}g \in \GL(\beta)\rbrace \] is called the space of semi-invariants on $\rep(Q,\beta)$ of weight $\chi$.

\begin{remark} Note that any $\theta \in \ZZ^{Q_0}$ defines a rational character $\chi_{\theta}:\GL(\beta) \to K^*$ by 
\begin{equation}
\chi_{\theta}((g(x))_{x \in Q_0})=\prod_{x \in Q_0}(\det g(x))^{\theta(x)}.
\end{equation}
In this way, we identify $\ZZ ^{Q_0}$ with $X^{\star}(\GL(\beta))$ whenever $\beta$ is a sincere dimension vector. In general, we have the natural epimorphism $\ZZ^{Q_0} \to X^{\star}(\GL(\beta))$. We also refer to the rational characters of $\GL(\beta)$ as (integral) weights of $Q$. 
\end{remark}

\begin{definition} Let $\theta$ be a weight of $Q$ and $V \in \rep(Q,\beta)$.
\begin{enumerate}
\item $V$ is called \emph{$\theta$-semi-stable} if there exists a semi-invariant $f \in \SI(Q,\beta)_{n \theta}$, with $n \geq 1$, such that $f(V) \neq 0$.

\item $V$ is called \emph{$\theta$-stable} if there exists a semi-invariant $f \in \SI(Q,\beta)_{n \theta}$, with $n \geq 1$, such that $f(V) \neq 0$,  $\dim \GL(\beta) V = \dim \GL(\beta)-1$, and the action of $\GL(\beta)$ on the principal open subset $\rep(Q,\beta)_f$ is closed.
\end{enumerate}
\end{definition}

In what follows, if $\gamma, \theta \in \ZZ^{Q_0}$, we define $\theta(\gamma):=\sum_{x \in Q_0} \theta(x) \gamma(x)$. We are now ready to state King's numerical criterion for semi-stability for quiver representations:

\begin{theorem} \cite[Theorem 4.1]{K} \label{King-general-thm} Let $V \in \rep(Q,\beta)$ and $\theta \in \ZZ^{Q_0}$ a weight of $Q$.
\begin{enumerate}
\item  $V$ is $\theta$-semi-stable if and only if $\theta(\dv V)=0$ and $\theta(\dv V^{\prime}) \leq 0$ for every subrepresentation $V' \leq V$.
\item  $V$ is $\theta$-stable if and only if $\theta(\dv V)=0$ and $\theta(\dv V^{\prime}) < 0$ for every proper subrepresentation $0 \neq  V'<V$.
\end{enumerate} 
\end{theorem}

\noindent
Given a weight $\theta$ of $Q$, we define $\rep(Q)^{ss}_{\theta}$ to be the full subcategory of $\rep(Q)$ consisting of all representations (including the zero representation) of $Q$ satisfying the list of homogeneous inequalities in Theorem \ref{King-general-thm}{(1)}. It turns out that $\rep(Q)^{ss}_{\theta}$ is an abelian subcategory of $\rep(Q)$ closed under extensions. Moreover, the simple objects of $\rep(Q)^{ss}_{\theta}$ are precisely the $\theta$-stable representations of $Q$.

If $\beta$ is a dimension vector of $Q$, $\rep(Q,\beta)^{ss}_{\theta}$ denotes the possibly empty (open) subset of $\rep(Q,\beta)$ consisting of $\theta$-semi-stable representations.

\subsection{Locally semi-simple quiver representations} \label{locally-semi-simple-sec} Let $Q$ be a quiver and $\beta$ a dimension vector of $Q$.

\begin{definition} \cite{Shm1} A representation $V \in \rep(Q, \beta)$ is said to be \emph{locally semi-simple} if the orbit $\SL(\beta)V$ is closed in $\rep(Q,\beta)$.
\end{definition}

In \cite{Kac}, Kac shows that any representation $W \in \rep(Q,\beta)$ has a Jordan decomposition of the form:
\[
W=V+N,
\]
where: 
\begin{itemize}
\item $V \in \rep(Q,\beta)$ is locally semi-simple; 

\item $N \in \rep(Q,\beta)$ is such that $\mathbf{0}_{\beta} \in \overline{\Stab_{\SL(\beta)}(V) \cdot N}$. Here $\mathbf{0}_{\beta}$ is the zero element of the vector space $\rep(Q,\beta)$. 
\end{itemize}

\noindent
So, the classification problem for quiver representations splits into two problems: (I) the classification of locally semi-simple quiver representations; and (II) the classification of nilpotent quiver representations with respect to certain algebraic subgroups of $\SL(\beta)$. 

The following result by Shmelkin provides a characterization of locally semi-simple quiver representations.

\begin{theorem} \cite{Shm1} \label{Shm-general-thm}
For a representation $V \in \rep(Q,\beta)$, the following statements are equivalent: 
\begin{enumerate}
\item $V$ is locally semi-simple;
\item there exists a semi-invariant $ 0 \neq f \in \SI(Q,\beta)_{\chi} $ such that $f(V) \neq 0$ and $\glb V$ is closed in $\rep(Q,\beta)_{f}$; 
\item there exists a character $\chi \in X^{*}( \glb)$ such that the orbit $\ker(\chi) V$ is closed in $\rep(Q,\beta)$.
\end{enumerate}
\end{theorem}

Combining King's results on semi-stability and Theorem \ref{Shm-general-thm}, we have: 

\begin{theorem} \label{locally-semi-simple-gen-theorem} Let $Q$ be a quiver, $\beta$ a dimension vector of $Q$, and $V \in \rep(Q,\beta)$. Let 
\[
V \simeq \bigoplus_{i=1}^r V_i^{m_i}
\]
be a decomposition of $V$ into pairwise non-isomorphic indecomposable representations $V_1, \ldots, V_r,$ with multiplicities $m_1, \ldots, m_r \geq 1$. Then the following are equivalent:
\begin{enumerate}
\item $V$ is locally semi-simple;

\item there exists a common weight $\theta$ of $Q$ such that each $V_i$ is $\theta$-stable. 
\end{enumerate}
\end{theorem}

\begin{proof}
($\Rightarrow$) We know from Theorem \ref{Shm-general-thm} that there exists a character $\chi \in X^{*}(\GL(\beta))$ such that $\ker(\chi) V$ is closed in $\rep(Q,\beta)$. Choose a weight $\theta \in \ZZ^{Q_0}$ such that $\chi=\chi_{\theta}$. We will show that each $V_i$ is $\theta$-stable. According to \cite[Propositions 2.6(i) and 3.2(i)]{K}, this is equivalent to checking that for every one parameter subgroup $\lambda \in X_{*}(\ker(\chi_{\theta}))$, $\lim_{t \to 0} \lambda(t)V \in \glb V$, whenever the limit exists. But this is clear since for any such $\lambda$, $\lim_{t \to 0} \lambda(t)V \in \overline{\ker(\chi_{\theta})V}=\ker(\chi_{\theta})V \subseteq \GL(\beta)V$.\\

\noindent
($\Leftarrow$) It follows from \cite[Proposition 3.2(i)]{K} that $\GL(\beta)V$ is closed in $\rep(Q,\beta)^{ss}_{\theta}$ and $V$ is $\theta$-semi-stable. Choose $f \in \SI(Q,\beta)_{N\theta}$ with $N \geq 1$ such that $f(V) \neq 0$. Then, $\GL(\beta)V$ remains closed in $\rep(Q,\beta)_f$, and so $V$ is locally semi-simple by Theorem \ref{Shm-general-thm}.
\end{proof}

\begin{remark} In light of Theorem \ref{locally-semi-simple-gen-theorem}, we say that a representation $V \in \rep(Q)$ is \emph{locally semi-simple} if there exists a weight $\theta \in \ZZ^{Q_0}$ such that the indecomposable direct summands of $V$ are all $\theta$-stable. 
\end{remark}

\begin{example} 
For a Dynkin quiver $Q$, any indecomposable representation $V \in \rep(Q)$ is locally semi-simple since any such $V$ is stable with respect to $\theta=\langle \dv V, \cdot \rangle - \langle \cdot, \dv V \rangle$.  
\end{example}

\begin{remark} Theorem \ref{locally-semi-simple-gen-theorem} tells us that for any locally semi-simple representation $V$ of $Q$, $\End_Q(V)$ is a semi-simple algebra. Indeed, let $V_1, \ldots, V_r$ be the pairwise non-isomorphic indecomposable direct summands of $V$ with multiplicities $m_1, \ldots, m_r \geq 1$. Then: (1) each $V_i$ is Schur since any stable representation is Schur; and (2) $\Hom_Q(V_i,V_j)=0$ for all $1 \leq i \neq j \leq r$ since a homomorphism between two stable representations is either zero or an isomorphism. Consequently, we get that
\[
\End_Q(V) \simeq \prod_{i=1}^r \Mat_{m_i \times m_i}(K),
\]
i.e. $\End_Q(V)$ is a semi-simple ring.
\end{remark}

\begin{definition} A sequence of representations $V_1, \ldots, V_r$ is said to be an \emph{orthogonal Schur sequence} if the representations $V_i$ are Schur and $\Hom_Q(V_i,V_j)=0$ for all $1 \leq i \neq j \leq r$. (Representations satisfying the second condition are called \emph{mutually orthogonal}.) 
\end{definition}

\begin{theorem} \label{ssalg}
Let $A$ be a $K$-algebra and $V$ an $A$-module. Let 
\[V \cong \bigoplus_{i=1}^{r} V_i^{m_i}\]
be a decomposition of $V$ into pairwise non-isomorphic indecomposable $A$-modules $V_1, \ldots, V_r$ with multiplicities $m_1, \ldots, m_r \geq 1$. Then $\text{End}_{A}(V)$ is a semi-simple $K$-algebra if and only if  $V_1, \ldots, V_r$ form an orthogonal Schur sequence. 
\end{theorem}

\begin{proof}
Let us write $\End_A(V)$ in block-matrix form:
\begin{equation*}  \displaystyle \text{End}_{A}(V) \cong \left(\text{Mat}_{m_{i} \times m_{j}} \left (\Hom_{A}(V_j, V_i) \right ) \right)_{i,j}.  
\end{equation*}

By \cite[ Lemma 1.3.3]{brion}, we know that:
\[ \End_A(V_i) = I_i \oplus K Id_{V_i},  \forall 1 \leq i \leq r,\]
where $I_i$ is a nilpotent ideal of $\End_A(V_i)$. This decomposition induces an algebra homomorphism $u_i:\Mat_{m_i \times m_i}\left( \End_A(V_i) \right) \to \Mat_{m_i\times m_i}(K)$ for each $1 \leq i \leq r$. Next, set up the function:
\[ u:\End_A(V) \to \prod_{i=1}^r \Mat_{m_i \times m_i}(K),\]
defined by $u((f_{i,j}))=(u_1(f_{1,1}), \ldots, u_r(f_{r,r}))$ for all $(f_{i,j}) \in \End_A(V)$. It is proved in \cite[Theorem 1.3.4]{brion} that $u$ is a surjective morphism of algebras with $\ker u$ a nilpotent ideal of $\End_A(V)$.\\
 
\noindent
($\Rightarrow$) If $\text{End}_{A}(V)$ is semi-simple then $\ker(u)=0$. This clearly implies that $\Hom_A(V_j, V_i) = 0$ for all $i \neq j$. Moreover, any nilpotent ideal of $\End_A(V_i)$ gives rise to a nilpotent ideal of $\text{End}_{A}(V)$. So we must have that $\text{End}_{A}(V_i) \cong K$ for each $1 \leq i \leq r$. \\

\noindent
($\Leftarrow$) If $ V_1, \ldots, V_r $ forms an orthogonal Schur sequence, then 
\[\End_{A}(V) \cong  \prod_i \Mat_{m_i \times m_i}(\Hom_{A}{(V_i, V_i)}) \cong \prod_{i} \Mat_{m_i \times m_i}(K). \] 
So $\text{End}_{A}(M)$ is a semi-simple ring by the Artin-Wedderburn Theorem.
\end{proof}

\section{The tame case} \label{tame-sec}
Kac's problem reduces to the following: For an orthogonal Schur sequence of representations of a tame quiver, can we find a common weight such that each representation is stable with respect to this weight?  

In \cite{DW1}, the authors use a generalization of orthogonal exceptional sequences to describe the faces of the so called  cones of effective weights associated to arbitrary (acyclic) quivers. Using their result, we are able to find a common stability weight for orthogonal exceptional sequences. It turns out that when at least one non-regular representation is present, orthogonal Schur sequences can be rearranged to form exceptional sequences (see Sections \ref{Dynkin-sec}, \ref{non-regular-sec}, and \ref{mixed-sec}). This technique fails for orthogonal Schur sequences of regular representations, so further analysis is needed in this case (see Section \ref{regular-sec}).

\begin{definition} \label{exceptional}
\begin{enumerate}
\item A representation $V$ is called \emph{exceptional} if $V$ is Schur and $\Ext_{Q}^{1}(V,V)=0$.  
\item 
Let $\mathcal{L}=(V_{1}, \ldots, V_r)$ be a sequence of exceptional representations. $\mathcal{L}$ is called an \emph{exceptional sequence} if $ \Hom_{Q}(V_{i}, V_{j})=\Ext_{Q}^{1}(V_{i}, V_{j}) = 0 \mbox{ for } i < j$.  If, in addition,  $\Hom_{Q}(V_{i}, V_{j})=0$ for all $i \neq j$, $\mathcal{L}$ is called an \emph{orthogonal exceptional sequence}. 
\end{enumerate}
\end{definition}

Let $\mathcal{L}=(V_1, \ldots, V_r)$ be an orthogonal exceptional sequence of representations of a quiver $Q$. If $\beta_1, \ldots, \beta_r$ are the dimension vectors of $V_1, \ldots, V_r$ then $(\beta_1, \ldots, \beta_r)$ is a quiver Schur sequence in the terminology of \cite{DW1}. According to Theorem 5.1 in \cite{DW1}, the sequence $(\beta_1, \ldots, \beta_r)$ corresponds to a unique face of the cone of effective weights associated to $Q$ and $\beta:=\beta_1+\ldots+\beta_r$. Moreover, if $\theta$ is a lattice point of $\mathcal F$ then each $\beta_i$ is $\theta$-stable. Since in our set-up $\beta_i$ is a real Schur root, the only, up to isomorphism, $\theta$-stable $\beta_i$-dimensional representation is $V_i$, $1 \leq i \leq r$. Hence, as an immediate consequence of \cite[Theorem 5.1]{DW1}, we get:

\begin{prop} \label{qsswt} Let $Q$ be a quiver and $\mathcal{L}=(V_1, \ldots, V_r)$ an orthogonal exceptional sequence of representations of $Q$. Then there exists a weight $\theta$ such that $V_i$ is $\theta$-stable for all $1 \leq i \leq r$.
\end{prop}

\begin{remark}
From Proposition \ref{qsswt} and Theorem \ref{locally-semi-simple-gen-theorem}, it follows that direct sums of representations whose dimension vectors form orthogonal exceptional sequences are locally semi-simple. Similar results can be found in \cite{Shm2}.

\end{remark}

\subsection{Dynkin quivers} \label{Dynkin-sec}

\begin{prop}\label{dynkin} Let $Q$ be a Dynkin quiver and $\mathcal{L} = (V_1, \ldots, V_r)$ an orthogonal Schur sequence. Then $\mathcal{L}$ can be arranged to form an orthogonal exceptional sequence. \end{prop}

\begin{proof} Let $T$ be the quiver with $T_{0}=\{1, \ldots, r\}$, and an arrow from $i$ to $j$ if and only if $\Ext^{1}_{Q}(V_{i}, V_{j}) \neq 0$.   We will show that $T$ is acyclic. Assume for a contradiction that there exists an oriented cycle $a_{1} \ldots a_{\ell}$ in $T$. Then: 
\begin{equation} \label{eq-ext}
\Ext_{Q}^{1}(V_{ta_{1}}, V_{ta_{2}}) \neq 0, \ldots,\Ext_{Q}^{1}(V_{ta_{\ell-1}}, V_{ta_{\ell}}) \neq 0, \Ext_{Q}^{1}(V_{ta_{\ell}}, V_{ta_{1}}) \neq 0.
\end{equation}

Let $\alpha_i = \dv V_i, 1 \leq i \leq r$. Then:
\begin{equation} \label{eq1} q(\alpha_{ta_{1}} + \ldots + \alpha_{ta_{\ell}}) = \sum _{i \neq j} \langle \alpha_{ta_{i}}, \alpha_{ta_{j}} \rangle + \sum_{i=1}^{\ell} q(\alpha_{ta_i}) = \sum _{i \neq j} \langle \alpha_{ta_{i}}, \alpha_{ta_{j}} \rangle + \ell. 
\end{equation}

We know that $\la \alpha_{ta_i}, \alpha_{ta_j} \ra \leq 0$ for any $i\neq j$ since the representations $V_i$ are mutually orthogonal. Hence:
\[
q(\alpha_{ta_{1}} + \ldots + \alpha_{ta_{\ell}}) \leq \sum_{i=1}^{\ell-1} \langle \alpha_{ta_{i}}, \alpha_{ta_{i+1}} \rangle + \langle \alpha_{ta_{\ell}}, \alpha_{ta_{1}}\rangle +\ell.
\]

Using $(\ref{eq-ext})$, we get that $\la \alpha_{ta_i}, \alpha_{ta_{i+1}} \ra \leq -1$ for all $1 \leq i \leq \ell$, where $ta_{\ell+1}=ta_1$. So, $q(\alpha_{ta_{1}} + \ldots + \alpha_{ta_{\ell}}) \leq  0$. But this impossible since $q$ is positive definite as $Q$ is assumed to be Dynkin. Thus $T$ is acyclic.

It is well known that when $T$ has no oriented cycles we can order the vertices of $T$ such that $ta > ha$ for each $a \in T_1$. In particular, if $i < j$, then $\Ext^1_Q(V_i, V_j)=0$ (otherwise there exists an arrow $i \rightarrow j$ in $T$, which would imply $i > j$). 
\end{proof}

\subsection{Euclidean quivers} Throughout this section we assume that $Q$ is a Euclidean quiver and denote by $\delta$ the unique imaginary Schur root of $Q$. 

Given a representation $V$, the \emph{defect} of $V$ is the constant $\langle \delta, \dv V \rangle$. A representation $V$ is called \emph{preprojective}, \emph{regular}, or \emph{preinjective} precisely when all indecomposable direct summands of $V$ have negative, zero, or positive defect respectively.

\begin{lemma} \label{homext} \cite{CB}  Let $X,Y \in \rep(Q)$ be indecomposable representations. 
\begin{enumerate}
	\item If $Y$ is preprojective and $X$ is not, then $\Hom_Q(X,Y) = 0$ and $\Ext_{Q}^{1}(Y,X)=0$.
	\item If $Y$ is preinjective and $X$ is not, then $\Hom_Q(Y,X)=0$ and $\Ext_{Q}^{1}(X,Y)=0$.
\end{enumerate}
\end{lemma}

\subsubsection{Non-regular case} \label{non-regular-sec}
In what follows, by ``non-regular representations of the same type'', we understand a collection of non-regular representations which are either all preprojective or all preinjective. 

\begin{prop} \label{allthesame} Let $\mathcal{L} =(V_1, \ldots, V_r)$ be an orthogonal Schur sequence of non-regular representations of the same type. Then $\mathcal{L}$ can be arranged to form an orthogonal exceptional sequence.
\end{prop}

\begin{proof}
 Let $T$ be the quiver with $T_{0}=\{1, \ldots, r\}$ and an arrow from $i$ to $j$ if and only if $\Ext^{1}_{Q}(V_{i}, V_{j}) \neq 0$. It suffices to show that $T$ is acyclic. Assume for a contradiction that there exists an oriented cycle $a_{1} \ldots a_{\ell}$ in $T$. Then:
 \[\dsp \Ext_{Q}^{1}(V_{ta_{1}}, V_{ta_{2}}) \neq 0, \ldots, \Ext_{Q}^{1}(V_{ta_{\ell-1}}, V_{ta_{\ell}}) \neq 0, \Ext_{Q}^{1}(V_{ta_{\ell}}, V_{ta_{1}}) \neq 0.\]  Let $\alpha_i = \dv V_i$, $1 \leq i \leq r$. In precisely the same way as the proof of Proposition \ref{dynkin}, we obtain:
\[q(\alpha_{ta_{1}} + \ldots + \alpha_{ta_{\ell}}) \leq 0.\]
This forces $q(\alpha_{ta_{1}} + \ldots + \alpha_{ta_{\ell}}) = 0$ since $q$ is positive semi-definite; in particular, \[ \alpha_{ta_{1}} + \ldots + \alpha_{ta_{\ell}}=r \delta \;\;\;\; (\mbox{for some } r \in \mbb{Z}_{>0}).\]  Thus:
\[ 0 = \la \delta, \alpha_{ta_{1}} + \ldots + \alpha_{ta_{\ell}}\ra = \la \delta, \alpha_{ta_{1}}\ra + \ldots +\la \delta, \alpha_{ta_{\ell}}\ra. \]
But this contradicts our assumption that all the $V_{i}$ have nonzero defect of the same sign. 
\end{proof}

Next we consider orthogonal Schur sequences of both preprojective and preinjective representations.
 
\begin{corollary} \label{mixedcomp1} Let $\mathcal{L}=(V_1, \ldots, V_r)$ be an orthogonal Schur sequence of non-regular representations. Then $\mathcal{L}$ can be arranged to form an orthogonal exceptional sequence.
\end{corollary}
 
\begin{proof} 
This follows from Proposition \ref{allthesame} and Lemma \ref{homext} by arranging the preprojectives to the left of the preinjectives. 
\end{proof}

\subsubsection{Mixed case} \label{mixed-sec}
Next we consider orthogonal Schur sequences that contain both non-regular and regular representations. It turns out that the assumption of mutual orthogonality greatly restricts these types of orthogonal Schur sequences.    
 
\begin{lemma} \label{regdelta} Let $X$ be a regular indecomposable representation with $\dv X = \delta$ and $Y$ a non-regular indecomposable representation. Then $X$ and $Y$ are not mutually orthogonal. 
\end{lemma} 

\begin{proof}
If $Y$ is preprojective then $Y$ has negative defect, thus $\la\delta, \dv Y \ra  = - \la \dv Y, \delta  \ra < 0$. This clearly implies that $\Hom_Q(Y, X) \neq 0$. The preinjective case is similar.
\end{proof}

\begin{lemma} \label{save} Let $X_{1}, \ldots, X_{\ell}$ be a collection of regular representations with $\alpha_{i} = \dv X_{i}$. Let $Y$ be a non-regular indecomposable representation. If $Y$ is mutually orthogonal to each $X_{i}$, then for any subcollection $X_{i_{1}}, \ldots, X_{i_{t}}$, the dimension vectors $\alpha_{i_1}, \ldots, \alpha_{i_t}$ satisfy:
\[\alpha_{i_1} + \ldots +  \alpha_{i_t} \neq r \delta, \;\;\;  \text{ for any } \; r \in \mbb{Z}.\] \end{lemma} 

\begin{proof} Assume for a contradiction that $\alpha_{i_1} + \ldots +  \alpha_{i_t} = r \delta$, with $r > 0$. If $Y$ is preinjective, then it has positive defect, so $0 < r\langle \delta , \dv Y \rangle $. Thus: 
\[ 0 <  \langle \alpha_{i_1} + \ldots +\alpha_{i_t}, \dv Y \rangle = \sum_{s=1}^{t} \langle \alpha_{i_s}, \dv Y \rangle=-\sum_{s=1}^{t} \dim_K \Ext^1_Q(X_{i_s},Y), \] 
which is impossible. The preprojective case is similar. 
\end{proof} 

\begin{lemma}\label{mixed1}
Let $\mathcal{L}=(V_1, \ldots, V_r)$ be an orthogonal Schur sequence consisting of both regular representations and non-regular representations of the same type. Then $\mathcal{L}$ can be arranged to form an orthogonal exceptional sequence. 
\end{lemma}

\begin{proof}
 Let $T$ be the quiver with $T_{0}=\{1, \ldots, r\}$ and an arrow from $i$ to $j$ if and only if $\Ext_{Q}^{1}(V_{i}, V_{j}) \neq 0$. It suffices to show that $T$ is acyclic. Suppose there exists an oriented cycle $a_{1} \ldots a_{\ell}$ in $T$. Then: \[\dsp \Ext_{Q}^{1}(V_{ta_{1}}, V_{ta_{2}}) \neq 0, \ldots, \Ext_{Q}^{1}(V_{ta_{\ell-1}}, V_{ta_{\ell}}) \neq 0,  \Ext_{Q}^{1}(V_{ta_{\ell}}, V_{ta_{1}}) \neq 0.\]

Let $\alpha_{i} = \dv V_{i}$, $1 \leq i \leq r$. Lemma \ref{regdelta} implies that the $\alpha_i$ are Schur roots with $\alpha_i \neq \delta, \forall 1 \leq i \leq r$. Therefore, they have to be real Schur roots; in particular, for each $i$,  $q(\alpha_{i})=1$. Using the same analysis as in Proposition \ref{allthesame}, we obtain $q(\alpha_{ta_{1}} + \ldots + \alpha_{ta_{\ell}}) = 0$. Thus $\alpha_{ta_{1}} + \ldots + \alpha_{ta_{\ell}}=r \delta$,
and it follows that
\[ 0 = \la \delta, \alpha_{ta_{1}} + \ldots + \alpha_{ta_{\ell}}\ra = \la \delta, \alpha_{ta_{1}}\ra + \ldots + \la \delta, \alpha_{ta_{\ell}}\ra \]
By Lemma \ref{save}, at least one $V_{ta_{i}}$ must be non-regular, so at least one $\la \delta, \alpha_{ta_{i}} \ra$ is nonzero, and the non-zero $\la \delta, \alpha_{ta_{i}} \ra$ are all of the same sign. This leads to a contradiction.
\end{proof}

\begin{prop} \label{mixedcomp} Let $\mathcal{L}=(V_1, \ldots, V_t)$  be an orthogonal Schur sequence of representations consisting of both regular and non-regular representations. Then $\mathcal{L}$ can be arranged to form an orthogonal exceptional sequence.
\end{prop}

\begin{proof}
We may assume $\mathcal{L}$ contains both preprojective and preinjective representations. By Proposition \ref{allthesame}, the preprojectives can be arranged to form an orthogonal exceptional sequence. By Lemma \ref{mixed1}, the regulars and preinjectives can be arranged to form an orthogonal exceptional sequence. By Lemma \ref{homext}, the combination of these two sequences (with the preprojectives first) forms an orthogonal exceptional sequence. 
\end{proof} 

We are now ready to consider orthogonal Schur sequences of regular representations.

\subsubsection{Regular case} \label{regular-sec}
In this section we restrict our attention to $\req =\rep{(Q)}_{\langle \delta, \cdot \rangle}^{ss}$, the abelian subcategory of regular representations of $\rep{Q}$ that is closed under extensions. Here, $\langle \delta, \cdot \rangle$ denotes the weight $\theta \in \ZZ^{Q_0}$ defined as $\theta(i)=\langle \delta, \mathbf{e}_i \rangle$ for all $i \in Q_0$.

In $\req$, there exist Schur representations that are not exceptional, and orthogonal Schur sequences that cannot be arranged to form orthogonal exceptional sequences, so different techniques will be required. 

We begin by recalling the Auslander-Reiten translate $\tau$ for hereditary algebras. For a more general definition, see \cite{AS-SI-SK}.
\begin{recall} Let $A=KQ$. The \emph{Auslander Reiten translations} are defined by $\tau(-):= D \Ext^{1}_{A}(-, A)$ and $\tau^{-}(-)=\Ext^{1}_{A}(D(-), A)$.    \end{recall}

\begin{definition} A regular representation is called \emph{regular simple} if it has no proper regular subrepresentations. \end{definition}

\begin{remark}
\begin{enumerate}
\item The regular simple representations are precisely the $\la \delta, \cdot \ra$-stable representations.
\item By ``regular non-simple," we mean regular representations which are not regular simple.
\end{enumerate}
\end{remark}

The image, kernel, and cokernel  of a morphism between two regular representations are regular. Also, if $X$ is an indecomposable regular representation, then $\tau^{i}(X) \neq 0$ for any $i \in \mbb{Z}$ and in fact, $X$ is $\tau$-periodic. 

\begin{lemma}\cite{CB} \label{reg-simple-properties}
Let $X$ be a regular simple representation. Then: 
\be[i)]
\item $X$ is Schur;
\item $\tau^{i}(X)$ is regular simple for all $i$;
\item $\tau(X) \cong X$ if and only if $\dv X$ is an imaginary root; 
\item if $X$ has period $p$, then $\dv X + \dv \tau(X) + \ldots + \dv \tau^{p-1}(X) = \delta$. 
\end{enumerate}
\end{lemma}

\begin{definition} A regular representation $X$ is called \emph{regular uniserial} if all of the regular subrepresentations of $X$ lie in a chain: 
\[ 0 = X_{0} \subsetneq X_{1} \subsetneq \ldots \subsetneq X_{r-1} \subsetneq X_{r} = X\]  In this case, $X$ has regular simple composition factors $X_{1}, X_{2}/X_{1}, \ldots, X_{r}/X_{r-1}$,  \emph{regular length} $r\ell(X):=r$, \emph{regular socle}  $\rSoc(X):=X_{1}$ and \emph{regular top} $\rTop(X):=X/X_{r-1}$. 
\end{definition} 

\begin{remark}
It follows from the proof of \cite[Theorem 2.12]{serg} that any regular subrepresentation of a regular indecomposable representation is also indecomposable. 
\end{remark}

\begin{theorem} \cite{CB} \label{runique}
Every indecomposable regular representation $X$ is regular uniserial. Moreover, if $E$ is the regular top of $X$, then the compositions factors of $X$ are precisely $E, \tau(E), \ldots, \tau^{\ell}(E)$ where $\ell+1=r\ell(X)$. Thus a regular indecomposable is uniquely determined by its regular top and regular length.
\end{theorem}

In light of Theorem \ref{runique}, for a regular indecomposable representation $V$, we write $V= \left (E, \tau(E), \ldots, \tau^{\ell}(E) \right )$ where $E$ is the regular top of $V$ and $\ell+1=r\ell(V)$. We will write these factors horizontally or vertically, depending on whatever is most convenient. 

\begin{remark} \label{AR}
Let $E$ be a non-homogeneous regular simple with period $p$ and let $ r \leq p -1 $. For $ 0 < j < r$, we always have:
\[ \left( \begin{tabular}{c} $\tau^{j}(E)$ \\ $\tau^{j+1}(E)$ \\$\vdots$ \\ $\tau^{r}(E)$  \end{tabular} \right) \hookrightarrow \left( \begin{tabular}{c} $E$ \\ $\tau(E)$ \\  $\vdots$ \\ $\tau^{r-1}(E) $\\$\tau^{r}(E)$  \end{tabular} \right) \;\;\; \mbox{ and }\;\;\;
\left (\begin{tabular}{c} $E$ \\ $\tau(E)$ \\$\vdots$ \\ $\tau^{r-1}(E)$\\$\tau^{r}(E)$  \end{tabular} \right ) \twoheadrightarrow \left (\begin{tabular}{c} $E$ \\ $\tau(E)$ \\  $\vdots$ \\$\tau^{j}(E)$  \end{tabular} \right ). \]
\end{remark}

\begin{lemma} \cite{CB} \label{dimcap} Let be  $X$ a regular indecomposable representation lying in a tube of period $p$. Then the following are equivalent: 
\begin{enumerate}[a)]
\item $X$ is Schur.
\item $\dv X \leq \delta$.
\item $r \ell(X) \leq p$. 
\end{enumerate}
\end{lemma}

\begin{definition}
Let $E$ be a regular simple with $\tau$-orbit $\mathcal{E}=\{E, \tau(E), \ldots, \tau^{p-1}(E)\}$. The \emph{tube $\mathcal{T}$ generated by $\mathcal{E}$} with period $p$ is the set of all indecomposable regular representations whose regular composition factors lie in $\mathcal{E}$. 
\end{definition}

If a tube $\mathcal{T}$ has period $p$, so does every indecomposable representation in $\mathcal{T}$.  Tubes generated by a $\delta$-dimensional regular simple are called \emph{homogeneous}. The number of non-homogeneous tubes of a Euclidean quiver is at most three.  

\begin{lemma}\label{soctop} Let $\mathcal{L}=( V_{1}, \ldots, V_{r})$ be an orthogonal Schur sequence of regular representations. Then for each $i \neq j$, the following hold:
\be[a)]
\item $\rSoc(V_{i}) \ncong \rSoc(V_{j})$;
\item  $\rSoc(V_{i}) \ncong \rTop(V_{j})$;
\item $\rTop(V_{i}) \ncong \rTop(V_{j})$.
\end{enumerate}
\end{lemma}

\begin{proof}
\leavevmode
\be[a)]
\item Suppose $\rSoc(V_{i})\cong \rSoc(V_{j})$. Since $V_{i} \not \cong V_{j}$, we know  $r\ell(V_{i}) \neq r\ell(V_{j})$. Without loss of generality, assume that  $r\ell(V_{i}) <  r\ell(V_{j})$. Then by Remark \ref{AR}, $V_{i} \hookrightarrow V_{j}$, so $\Hom_Q(V_{i}, V_{j}) \neq 0$ (contradiction). (c) is similar.

\item If $\rSoc(V_{i}) \cong $\rTop$(V_{j})$ then:
\[ V_{j} \twoheadrightarrow \rTop(V_{j})\cong \rSoc(V_{i}) \hookrightarrow V_{i} \] Thus $\Hom_Q(V_{j}, V_{i}) \neq 0$, which is a contradiction. 
\end{enumerate}
\end{proof}

\begin{lemma}\label{topmap}
Let $V$ and $W$ be regular indecomposable representations lying in the same tube. If \mbox{$f:V \rightarrow W$} is a nonzero morphism then the regular top of $V$ is isomorphic to a regular composition factor of $W$.   
\end{lemma}

\begin{proof}
Write $V=(T, \ldots, \tau^{\ell}(T))$. Then $\im f \cong V/\ker{f} \leq W$. If $\ker{f}=0$, we are done. Otherwise, $\ker{f} \cong
(\tau^k(T), \ldots, \tau^{\ell}(T))$ for $ 0 < k < \ell$. In particular, $V/\ker{f} \cong (T, \ldots, \tau^{k-1}(T)) \leq W$, i.e. $T$ is the regular top of a regular subrepresentation of $W$. In particular, $T$ is a regular composition factor of $W$.
\end{proof}

Let $V= \left( E, \ldots, \tau^{\ell}(E)\right)$ and $W$ be two regular indecomposable representations lying in the same tube. We write $W \sqsubset V$ if:
\[ \rTop(W) \cong \tau^j(E) \text{~and~} \rSoc(W) \simeq \tau^i(E) \text{~with~} 0<j \leq i <\ell.
\]
We will also write $V \sqcap W$ to mean the set of regular composition factors shared by $V$ and $W$. Using this notation, we have the following lemma: 

\begin{lemma} \label{orthog}
Let $V_{1}$ and $V_{2}$ be regular Schur representations lying in the same tube. Then $V_{1}$ and $V_{2}$ are mutually orthogonal if and only if either $V_{1} \sqcap V_{2} = \varnothing$, $V_{1} \sqsubset V_{2}$,  or  $V_{2} \sqsubset V_{1}$.  
\end{lemma} 

\begin{proof}
$(\Rightarrow)$ If $V_{1} \sqcap V_{2} = \varnothing$, we are done. Otherwise, there exists $E \in  V_{1} \sqcap V_{2}$ and we can write: 
\begin{align*}
V_{1} & = (\tau^{-t_{1}}(E), \ldots, E, \ldots, \tau^{s_{1}}(E) ) \\
V_{2} & = (\tau^{-t_{2}}(E), \ldots, E, \ldots, \tau^{s_{2}}(E) )
\end{align*}
Since $V_{1}$ and $V_{2}$ are mutually orthogonal,  Lemma \ref{soctop} implies  $\rSoc(V_{1})$, $\rSoc(V_{2})$, $\rTop(V_{1})$, $\rTop(V_{2})$ are all distinct. So, without loss of generality, let us assume $s_{2} > s_{1}$. We will show that $V_1 \sqsubset V_2$.  It suffices to show $t_{1} < t_{2}$. If this is not the case, then by Remark \ref{AR} we have the following morphisms:
\[ V_{2} \twoheadrightarrow \left (\begin{tabular}{c} $\tau^{-t_{2}}(E)$ \\ $\vdots$ \\ $\tau^{s_{1}}(E)$  \end{tabular} \right ) \hookrightarrow V_{1}\] which would contradict orthogonality. Hence, $V_1 \sqsubset V_2$.\\

\noindent
$(\Leftarrow)$ If $V_1 \sqcap V_2 = \varnothing$ then Lemma \ref{topmap} implies that $V_1$ and $V_2$ are mutually orthogonal. 

Now, let us suppose $V_1 \sqsubset V_2$ (the case when $V_2 \sqsubset V_1$ is similar). Let $E$ be the regular top of $V_2$ and $\ell_2+1$ its regular length. Then, by definition, there exist $0<j\leq i <\ell_2$ such that:
\begin{itemize}
\item $\rTop(V_1)=\tau^j(E)$;
\item $\rSoc(V_1)=\tau^i(E)$;
\item the composition factors of $V_1$ are of the form $\tau^k(E)$ with $j \leq k \leq i$;
\item the regular length of $V_1$ is $i-j+1$.
\end{itemize}

Let us prove first that $\Hom_Q(V_2,V_1)=0$. It this were not the case, then the regular top of $V_2$ would be isomorphic to a composition factor of $V_1$ by Lemma \ref{topmap}, i.e. $E$ would be isomorphic to $\tau^k(E)$ with $0<k<\ell_2$. But this is a contradiction since the period of $E$ is strictly larger than $\ell_2$ as $V_2$ is assumed to be a regular Schur representation.

To check that $\Hom_Q(V_1,V_2)=0$, we will work with the regular socle of $V_1$. Specifically, let us assume for a contradiction that there exists a non-zero morphism $f \in \Hom_Q(V_1,V_2)$. Then, using Remark \ref{AR}, we get that $\rSoc(V_1/\ker{f})=\tau^l(\rTop(V_1))=\tau^{l+j}(E)$ for some $0 \leq l \leq i-j$. On the other hand, we have that $\rSoc(\im{f})=\rSoc(V_2)=\tau^{\ell_2}(E)$, and so $\tau^{\ell_2}(E) \cong \tau^{l+j}(E)$ (contradiction).
\end{proof}

\begin{definition} Let $\mathcal{L}$ be an orthogonal Schur sequence of regular representations lying in a tube $\mathcal T$. We say that $\mathcal{L}$ is \emph{maximal} if every regular simple of $\mathcal{T}$ is either an element of $\mathcal{L}$, or the regular top or regular socle of a (unique) representation in $\mathcal{L}$.
\end{definition} 

\begin{remark} Let $\mathcal T$ be a tube and $\mathcal L \subseteq \mathcal T$ an orthogonal Schur sequence. Then we can always extend $\mathcal L$ to a maximal orthogonal Schur sequence by simply adding to $\mathcal L$ the regular simples of $\mathcal T \snot \mathcal L$.
\end{remark}

It turns out that working with maximal orthogonal Schur sequences greatly simplifies the task of solving for a common stability weight. In what follows, we restrict our attention to maximal sequences. 

\begin{corollary} \label{clarity}
Let $\mathcal{L}$ be a maximal orthogonal Schur sequence lying in a non-homogenous tube $\mathcal{T}$ of period $p$. Let $V \in \mathcal{L}$ with $\rTop(V)=T$ and $r\ell(V)=\ell+1$, i.e. $V = (T, \tau(T), \ldots, \tau^{\ell }(T))$. Then for any $ 0 < k < \ell $, one of the following three cases holds:
\begin{enumerate}
\item $\tau^{k}(T) \in \mathcal{L}$; 
\item There exists a unique $X \in \mathcal{L}$ and $ k < j < \ell $ such that $\tau^k(T)=\rTop(X)$ and $\tau^j(T)=\rSoc(X)$; 
\item There exists a unique $X \in \mathcal{L}$ and $ 0 < j < k < \ell  $ such that $\tau^k(T)=\rSoc(X)$ and $\tau^j(T)=\rTop(X)$. 
\end{enumerate}
\end{corollary}

\begin{proof} Suppose $\tau^k(T) \not \in \mathcal{L}$. Since $\mathcal{L}$ is maximal, we know there exists $X \in \mathcal{L}$ such that $\tau^k(T)$ is either the regular top or the regular socle of $X$. Either way $V \not \sqsubset X$, so  Lemma \ref{orthog} implies $X \sqsubset V$. The uniqueness part of our claim follows from Lemma \ref{soctop}.
\end{proof} 

Next we show that to obtain a stability weight for a regular representation, we only need to check King's criterion for the regular subrepresentations. 

\begin{lemma}\label{regularstable12} Let $\mathcal{L}=(V_1, \ldots, V_r)$  be a collection of regular representations and $\theta$ a weight such that each $V_{i}$ is $\theta$-stable in $\req$, i.e. for each $i$, $\theta( \dv V_{i})=0$ and $\theta(\dv V_{i}^{\prime}) < 0 $ for all proper regular subrepresentations $0 \neq V'_i<V_{i}$.  Then there exists a weight $\sigma$ such that each $V_{i}$ in $\mathcal{L}$ is $\sigma$-stable in $\rep{Q}$. 
\end{lemma}

\begin{proof}
A non-regular subrepresentation of a regular indecomposable is either preprojective or a direct sum of regular and preprojective subrepresentations. In particular, any non-regular subrepresentation has negative defect.  

Let $N = \max \{ \theta(\dv X_i^{\prime}) | X_i^{\prime} \leq V_{i} \mbox{ is non-regular }, V_{i} \in \mathcal{L} \}$ and define 
\[\sigma =  \begin{cases} \theta & \mbox{if } N < 0\\
					\theta + \langle \delta, \cdot \rangle & \mbox{if } N =0\\
					\theta + (N +1)\langle \delta, \cdot \rangle & \mbox{if } N > 0\\
			\end{cases}\]
Choose $V_i \in \mathcal{L}$. We will show $V_{i}$ is $\sigma$-stable. If $N<0$, there is nothing to show.  Suppose  $N > 0$.  First, observe that $\sigma(\dv V_i) = \theta(\dv V_i ) + (N+1) \langle \delta, \dv V_i \rangle = 0$. 
For a regular subrepresentation  $V^{\prime}_{i} \leq V_i$, we have:
\[ \sigma(\dv V^{\prime}_i) = \theta(\dv V_i^{\prime}) + (N+1) \langle \delta, \dv V^{\prime}_i \rangle = \theta(\dv V_i^{\prime}) < 0\]
For a non-regular subrepresentation $X^{\prime}_i \leq V_i$, we have  
\[ \sigma(\dv X^{\prime}_i) = \theta(\dv X_i^{\prime}) + (N+1) \langle \delta, \dv X^{\prime}_i \rangle \leq  \theta(\dv X_i^{\prime}) - (N +1) < 0\]
Thus, each $V_i$ is $\sigma$-stable by Theorem \ref{King-general-thm}. The case when $N=0$ is similar. 
\end{proof}

\begin{rmk} We point out that for a homogeneous Schur representation $V$ to be $\theta$-stable in $\req$, we only need to check that $\theta(\delta)=0$. This because for such a $V$, $\dv V = \delta$ and its regular subrepresentations are $\{0\}$ and $V$ by Lemma \ref{dimcap}.
\end{rmk}

\begin{lemma} \label{depends} \cite[Lemma 6.1]{IngPacTom}
Let  $\mathcal{T}_{1}, \ldots, \mathcal{T}_{N}$ be the non-homogeneous tubes of $Q$, with $p_{i}$ the period of $\mathcal{T}_{i}$.  Let $\{\epsilon_{ij}\}$ be the set of dimension vectors of all non-homogeneous regular simple representations, with $1 \leq i \leq N$ and $ 1 \leq j \leq p_{i}$.  Suppose there exists constants $c_{ij}$ such that $\dsp \sum_{i,j} c_{ij} \epsilon_{ij} = 0$.  Then for each $i$, $\dsp C_{i}:=c_{i1}=\ldots =c_{ip_{1}}$. Furthermore, $\dsp C_{1} + \ldots + C_{N} = 0$.  
\end{lemma} 

While the next corollary is stated for the case of three non-homogenous tubes, similar statements can be made in the other cases.

\begin{corollary} \label{crosstube} Suppose there are exactly three non-homogenous tubes and let \[ \{E_1, \ldots, E_{p_1}\}, \{L_{1}, \ldots, L_{p_2}\}, \{K_1, \ldots, K_{p_3}\}\] be the regular simples of each of these tubes. Then the vectors 
\[\dv E_1, \ldots, \dv E_{p_1}, \dv L_{1}, \ldots, \dv  L_{p_2 - 1}, \dv K_1, \ldots, \dv K_{p_3 - 1} \] are linearly independent over $\mbb{R}$.  
\end{corollary}

\begin{proof}
Suppose $\sum_{i=1}^{p_1} a_i \, \dv E_{i} + \sum _{j=1}^{p_2 -1} b_{j} \, \dv L_{j} + \sum_{s=1}^{p_3 -1} c_{s} \, \dv K_{s} = 0$. By Lemma \ref{depends}, there exist $a,b,c \in \mbb{R}$ such that for each $i, j, s$ 
\[ a : = a_i, b:=b_j, \mbox{ and } c:=c_s \] Furthermore, we can write $\sum _{j=1}^{p_2 -1} b \, \dv L_{j} = b(\delta - \dv L_{p_2})$ and $\sum_{s=1}^{p_3 -1} c \, \dv K_{s} = c(\delta - K_{p_3})$. Putting everything together gives:
\begin{equation} \label{terry}  (a + b + c) \delta = b \,\dv  L_{p_2} + c \, \dv  K_{p_3}. \end{equation}  Note that indecomposable representations lying in different tubes are always mutually orthogonal and do not have any non-trivial extensions. Thus applying the quadratic form to (\ref{terry}) gives:
\[ 0 = b^{2} + c^{2} +0 + 0. \] 
So $b=c=0$, and thus $a=0$. 
\end{proof}

\begin{prop} \label{sol1} Let $\mathcal{L}$ be a maximal orthogonal Schur sequence lying in a non-homogenous tube $\mathcal{T}$. Then there exists a weight $\theta$ such that each representation in $\mathcal{L}$ is $\theta$-stable. 
\end{prop} 

\begin{proof} Let $E_{0}, E_1 = \tau(E_0), \ldots, E_{p-1}=\tau^{p-1}(E_0)$ be the regular simples of $\mathcal{T}$ and assume $\mathcal{L}$ is maximal. For each $V_{i} \in \mathcal{L}$ which is regular non-simple, let $F_{i}:=\rSoc(V_{i})$ and $T_{i}:=\rTop(V_{i})$. By Lemma \ref{depends}, we can solve the following system of $p$ inequalities for $\theta$:
\begin{equation} \label{thomas} 
\begin{cases}
\theta(\dv F_{i} ) < 0 \\
\theta(\dv T_{i} ) = -\theta(\dv F_{i})\\
\theta(\dv E_i)=0 & (\text{ for } E_{i} \in \mathcal{L} ) 
\end{cases}
\end{equation}

Choose $V \in \mathcal{L}$ with $\ell+1=r\ell(V)$. If $V$ is regular simple then $\theta(\dv V)=0$ by the construction of $\theta$ and hence $V$ is $\theta$-stable by Lemma \ref{regularstable12}.

Let us assume now that $V$ is not simple in $\mathcal{R}(Q)$. We will show $V$ is $\theta$-stable by checking that $\theta(\dv V)=0$ and $\theta(\dv V')<0$ for any regular subrepresentation $0 \neq V' <V$. 

Let $T=\rTop(V)$, $S=\rSoc(V)=\tau^{\ell}(T)$, and $V'$ a proper regular subrepresentation of $V$. Then:
\[\dv V = \dv T + \dv \tau(T) + \ldots +  \dv \tau^{\ell}(T),\]
and 
\[ \dv V^{\prime} = \dv \tau^{i}(T) + \ldots + \dv \tau^{\ell }(T),\]
for some $0<i \leq \ell$.

Let $1< k< \ell $. If $\tau^k(T)$ is not a regular socle or regular top for any regular non-simple representation in $\mathcal{L}$, then by construction 
\[\theta(\dv \tau^k(T))=0.\]

If  $\tau^{k}(T)$ is the regular top or socle of some regular non-simple representation X in $\mathcal{L}$, Corollary \ref{clarity} implies there exists a unique $j$ with  $k < j <\ell  $ such that $\{\tau^k(T), \tau^{j}(T)\}=\{\rSoc(X), \rTop(X) \}$. Thus, by the construction of $\theta$ in (\ref{thomas}), we have:
\begin{equation} \label{socletopsum} 
\theta(\dv \tau^k(T)) + \theta(\dv \tau^j(T))=\theta(\rTop(X)) + \theta(\rSoc(X))=0.
\end{equation}

Putting everything together, we get that:
\[\theta(\dv V)=\theta(\dv T)+0+\ldots+0+\theta(\dv S)=0.\]

Finally, we determine the sign of $\theta(\dv V')$ by analyzing the composition factors of $V'$ . If $i=1$ then
\[ \theta(\dv V-\dv T)=-\theta(\dv T) < 0.\]
Now, let us assume that $i>1$. Let $F$ be a composition factor of $V'$. According to Corollary \ref{clarity}, exactly one of the following cases occur.
\begin{enumerate}
\item [(a)] $F$ is not a regular socle or top for any regular non-simple representation in $\mathcal L$. Then, $F$ is one of the $E_i$'s, and hence \[\theta(\dv F)=0.\]

\item [(b)] There exists a unique composition factor $F'$ of $V'$ such that $F$ and $F'$ are the regular top and socle of some representation $X$ in $\mathcal L$. In this case, by the construction of $\theta$, we have:
\[\theta(\dv F)+\theta(\dv F')=0.\]

\item [(c)] $F$ is one of the $F_i$'s, being the regular socle of some representation $Y \in \mathcal L$ whose regular top is not among the composition factors of $V'$. Then: 
\[\theta(\dv F)<0.\]
\end{enumerate} 

In conclusion, we get that:
\[\theta(\dv \tau^i(T)) + \ldots + \theta(\dv S) \leq \theta(\dv S)<0. \]
The proof now follows from Lemma \ref{regularstable12}.
\end{proof} 

\begin{prop} \label{regcomp}
Let $\mathcal{L}$ be an orthogonal Schur sequence of regular representations. Then there exists a common stability weight for $\mathcal{L}$. 
\end{prop}

\begin{proof}
There are at most three non-homogeneous tubes $T_{1}, T_{2}$, and $T_{3}$ with periods $p_{1}, p_{2}$, and $p_{3}$ respectively. Assume $\mathcal{L}$ is an orthogonal Schur sequence such that:
\[\mathcal{L}=\mathcal{L}_0 \cup \mathcal{L}_{1} \cup \mathcal{L}_{2} \cup \mathcal{L}_{3} ,\]
where $\mathcal{L}_0$ consists of only homogeneous Schur representations, and $\mathcal{L}_i \subseteq T_{i}$, $1 \leq i \leq 3$. Extend $\mathcal{L}_{i}$ to a maximal orthogonal Schur sequence in $\mathcal{T}_i$ for each $1 \leq i \leq 3$. 

It suffices to show we can solve the system of inequalities (\ref{thomas}) across each tube simultaneously. Notice that when we satisfy (\ref{thomas}) over one tube, we automatically have $\theta(\delta) = 0$. Thus, we have one less equation to solve in the other tubes. So by Corollary \ref{crosstube}, we can solve (\ref{thomas})  across the non-homogenous tubes simultaneously.   Any other representation in our orthogonal Schur sequence must lie in a homogeneous tube, and thus has dimension $\delta$. But we already have $\theta(\delta)=0$. So far, we have proved that there exists a weight $\theta$ such that for any $V \in \mathcal L$ and $V' \leq V$ with $V'$ regular, $\theta(\dv V)=0$ and $\theta (\dv V') < 0$. 

Finally, invoking Lemma \ref{regularstable12}, it follows that there exists a common weight $\theta_0$ such that any $V \in \mathcal L$ is $\theta_0$-stable. 
\end{proof}

\begin{example}
Let $Q$ be the $\widetilde{\mathbb{D}}_5$ quiver: 
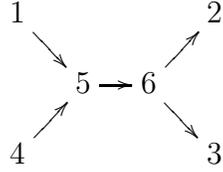
\[ \vci{\xymatrix@R=1em@C=1em{1 \ar[dr] & \; & \; &2 \\
\; &5 \ar[r] & 6 \ar[ur]  \ar[dr]  & \; \\
  4 \ar[ur] & \; & \; &3  \\}}
\] 

The three non-homogeneous regular tubes of $Q$ are generated by the following regular simples:

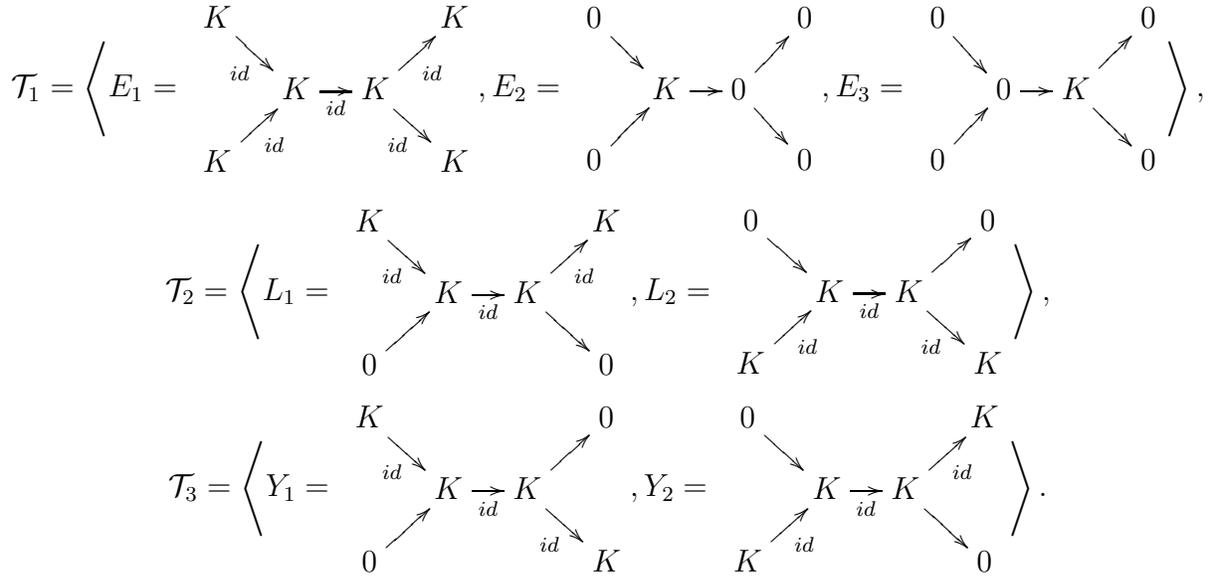
\begin{equation*} 
\mathcal{T}_1 = \left < E_1 = \vci{ \xymatrix@R=1em@C=1em{K  \ar[dr]_{id} & \; & \; & K \\
\; &K  \ar[r]_{id} & K \ar[ur]_{id}  \ar[dr]_{id}  & \; \\
  K \ar[ur]_{id} & \; & \; & K  \\}}, E_2 =  \vci{ \xymatrix@R=1em@C=1em{0  \ar[dr] & \; & \; & 0 \\
\; &K  \ar[r] & 0 \ar[ur]  \ar[dr]  & \; \\
  0 \ar[ur] & \; & \; & 0  \\}}, E_3=\vci{ \xymatrix@R=1em@C=1em{0  \ar[dr] & \; & \; & 0 \\
\; &0   \ar[r] & K \ar[ur]  \ar[dr]  & \; \\
  0 \ar[ur] & \; & \; & 0  \\}}  \right >,
\end{equation*}

\begin{equation*}
\mathcal{T}_2 = \left <L_1 = \vci{ \xymatrix@R=1em@C=1em{K  \ar[dr]_{id} & \; & \; & K \\
\; &K  \ar[r]_{id} & K \ar[ur]_{id}  \ar[dr]  & \; \\
 0 \ar[ur] & \; & \; & 0  \\}}, L_2 =  \vci{ \xymatrix@R=1em@C=1em{0  \ar[dr] & \; & \; & 0 \\
\; &K  \ar[r]_{id} & K \ar[ur]  \ar[dr]_{id}  & \; \\
 K \ar[ur]_{id} & \; & \; & K  \\}}  \right >, 
\end{equation*}

\begin{equation*}
\mathcal{T}_3 = \left <Y_1 = \vci{ \xymatrix@R=1em@C=1em{K  \ar[dr]_{id} & \; & \; & 0 \\
\; &K  \ar[r]_{id} & K \ar[ur]  \ar[dr]_{id}  & \; \\
 0 \ar[ur] & \; & \; & K  \\}},Y_2 =  \vci{ \xymatrix@R=1em@C=1em{0  \ar[dr] & \; & \; & K \\
\; &K  \ar[r]_{id} & K \ar[ur]_{id}  \ar[dr]  & \; \\
 K \ar[ur]_{id} & \; & \; &0   \\}}  \right >. 
\end{equation*}

Consider the following orthogonal Schur sequence
\[\mathcal L=\mathcal L_0 \cup \mathcal L_1 \cup \mathcal L_2 \cup \mathcal L_3, \]
where:
\[\mathcal L_0= \left \{  V_0=  \vci{ \xymatrix@R=1em@C=1em{K  \ar[dr]^{\left [\begin{smallmatrix} 1 \\ 0 \end{smallmatrix} \right ]} & \; & \; & K \\
\; &K^2  \ar[r]_{id} & K^2 \ar[ur]^{[1 \,\, 1]}  \ar[dr]_{[1 \,\, 2 ]}  & \; \\
 K \ar[ur]_{\left [\begin{smallmatrix} 0 \\1 \end{smallmatrix} \right ]} & \; & \; & K  \\ }} \right  \},\]
 
\[\mathcal L_1= \left \{V_1 = \vci{ \xymatrix@R=1em@C=1em{K  \ar[dr]^{\left [\begin{smallmatrix} 1 \\ 0 \end{smallmatrix} \right ]} & \; & \; & K \\
\; &K^2  \ar[r]_{id} & K^2 \ar[ur]^{[1 \,\, 1]}  \ar[dr]_{[1 \,\, 1]}  & \; \\
 K \ar[ur]_{\left [\begin{smallmatrix} 0 \\1 \end{smallmatrix} \right ]} & \; & \; & K  \\}} = \left( \begin{tabular}{c} $E_1$ \\ $E_2$ \\ $E_3$ \end{tabular} \right ), V_2 = E_2 \right \}, \]
 
 \[\mathcal L_2= \left\{V_{3} =
 \vci{ \xymatrix@R=1em@C=1em{K  \ar[dr]^{\left [\begin{smallmatrix} 1 \\ 0 \end{smallmatrix} \right ]} & \; & \; & K \\
\; &K^2  \ar[r]_{id} & K^2 \ar[ur]^{[1 \,\, 0]}  \ar[dr]_{[1 \,\, 1]}  & \; \\
 K \ar[ur]_{\left [\begin{smallmatrix} 0 \\1 \end{smallmatrix} \right ]} & \; & \; & K  \\ }}= \left( \begin{tabular}{c} $L_1$ \\ $L_2$  \end{tabular} \right )
 \right \}, \text{~and~}\mathcal L_3=\{V_4=Y_1, V_5=Y_2\}.
\]



We first solve for a weight $\theta$  so that each representation $V_i$ is $\theta$-stable in $\mathcal{R}(Q)$, then modify $\theta$ to obtain a stability weight in $\rep{Q}$. We begin by solving the following system:

\[ \left [\begin{array}{c} 
\dv E_1 \\
\dv E_2 \\
\dv E_3 \\
\dv L_1 \\
\dv Y_1 
\end{array} \right ] \cdot \theta
= 
 \left [\begin{array}{cccccc}
1 & 1 & 1 & 1 & 1 & 1 \\
0 & 0 & 0 & 0 & 1 & 0 \\
0 & 0 & 0 & 0 & 0 & 1 \\
1 & 1 & 0 & 0 & 1 & 1 \\
1 & 0 &1 & 0 & 1 & 1 \\
\end{array} \right ] 
\left[
\begin{array}{c}
\theta_1 \\
\theta_2 \\
\theta_3 \\
\theta_4 \\
\theta_5 \\
\theta_6     
\end{array} \right]
=
\left [\begin{array}{c}
1 \\
0 \\
-1 \\
1 \\
0
\end{array} \right ].
\]
(Note that the constants on the right hand side of the equations of the system above are picked according to the recipe described in the proof of Proposition \ref{sol1}.)

The general solution of this system is $(t, 2-t, 1-t, t-1, 0, -1)$ for $t \in \mbb{R}$. When $t=1$, we get  $\theta = (1,1,0,0,0,-1)$ and it is easy to check that each $V_i$ is $\theta$-stable in $\mathcal{R}(Q)$. (Of course, any other integer $t$ and the corresponding $\theta$ work equally well.) Furthermore, if we let $N=\max\{\theta(\dv X_i)  \mid X_i \leq V_i, X_i \mbox{ non-regular}  \}$, then $N=1$. 

Now set:
\[\sigma= \theta + 2\langle \delta, \cdot \rangle = (3,-1,-2,2,0,-1).\] 
One can check that each $V_i$ is $\sigma$-stable. In particular, it follows that the representation $\displaystyle V = \bigoplus_{i=1}^{6} V_i$ is locally semi-simple. 
\end{example}

\section{Proof of Theorem \ref{main-thm}} \label{proof-thm}
We begin with an example of a representation of a wild quiver that has a semi-simple endomorphism ring but is not locally semi-simple. 
\begin{example} \label{K3}
Let $Q$ be the 3-arrow Kronecker quiver  
\[ \vci{\xymatrix{ \bullet \ar@<-1.ex>[r] \ar[r] \ar@<1.ex>[r] & \bullet }}\] and consider the representation $V$ defined by $V(1)=V(2)=K^2$, and maps $ \displaystyle \left[\begin{smallmatrix} 1 & 0 \\ 0 & 1  \end{smallmatrix} \right], \; \left [ \begin{smallmatrix} 1 & 0  \\ 0 & 0  \end{smallmatrix} \right],   \left [ \begin{smallmatrix} 0 & 0  \\ 0 & 1  \end{smallmatrix} \right] $. 

It is easy to check that $V$ is a Schur representation.  Suppose $V$ is $\theta=(\theta_{1}, \theta_{2})$-stable for some weight $\theta$. Then $\theta(1)+\theta(2)=0$ as $\theta((2,2))=0$. Moreover, it is easy to see that $V$ has a proper subrepresentation of dimension vector $(1,1)$. This would imply that $\theta(1)+\theta(2)<0$ (contradiction). So, $V$ cannot be stable with respect to any weight.
\end{example}

Let $Q'$ and $Q$ be two acyclic quivers and let $\mathcal F: \rep{Q'} \rightarrow \rep{Q}$ be a fully faithful exact embedding. For a dimension vector $\beta^{\prime} \in \mbb{Z}_{\geq 0}^{Q'_0}$ of $Q'$, set \[\mathcal F(\beta'):=\sum_{i \in Q_{0}^{\prime}} \beta'(i)\dv \mathcal F(S_i) \in \mbb{Z}_{\geq 0}^{Q_{0}}.\]
Then for any $V^{\prime} \in \rep(Q')$ with $\dv V^{\prime} = \beta^{\prime}$, we have $\dv \mathcal F(V^{\prime})=\mathcal F(\beta')$. 

For a weight $\theta \in \mbb{Z}^{Q_{0}}$ of $Q$, set 
\[\mathcal F^{-1}(\theta):=\left ( \theta(\dv \mathcal F(S_i)) \right)_{i \in Q_0'} \in \mbb{Z}^{Q'_0}.\]
Then, for any dimension vector $\beta'$ of $Q'$ and any weight $\theta$ of $Q$, we have

\[\theta(\mathcal F(\beta'))=\theta \left (\sum_{i \in Q_0'} \beta' \dv \mathcal F(S_i) \right) = \sum_{i \in Q'_0} \beta'(i)\theta(\dv \mathcal F(S_i)) = \mathcal F^{-1}(\theta)(\beta').\]

\begin{prop} \label{refstab} Given any wild quiver $Q$, there exists a Schur representation that is not stable for any weight. 
\end{prop}

\begin{proof} Since the path algebra of any wild quiver is strictly wild, we know that there exists a fully faithful exact embedding $\mathcal{F}:\rep(K_3) \to \rep(Q)$. 

Let $V^{\prime}$ be the representation of $K_{3}$ in Example \ref{K3} with $\beta' = \dv V'=(2,2)$. If $V$ denotes $\mathcal{F}(V^{\prime})$ then its dimension vector is $\beta=\mathcal F(\beta')$. 

Assume for a contradiction that $V\in \rep{Q}$ is $\theta$-stable for some weight $\theta \in \ZZ^{Q_0}$. In particular, we have that $\theta(\beta)=0$. Denoting $\mathcal F^{-1}(\theta)$ by $\theta'$, we also get that $\theta'(\beta')=0$.

Now, let $W'$ be the subrepresentation of $V'$ of dimension vector $(1,1)$ from the example above. Since $\mathcal{F}$ is a fully faithful exact functor, $W:=\mathcal{F}(W')$ is a proper subrepresentation of $V$. Moreover, $\theta'(\dv W')=\theta(\dv W)<0$ but this is a contradiction since $\theta'(\dv W')={1 \over 2} \theta'(\beta')=0$.   
\end{proof}  

Now, we are ready to prove Theorem \ref{main-thm}:

\begin{proof}[Proof of Theorem \ref{main-thm}]
($\Leftarrow$)  This follows from Proposition \ref{refstab}.\\

\noindent
($\Rightarrow$)
Assume that $Q$ is a tame quiver and let $V$ be a representation of $Q$ such that $\End_Q(V)$ is a semi-simple ring. Let $V_1, \ldots, V_r$ be the pairwise non-isomorphic indecomposable direct summands of $V$. According to Theorem \ref{locally-semi-simple-gen-theorem}, we need to show that there exists a common weight $\theta \in \ZZ^{Q_0}$ such that each $V_i$ is $\theta$-stable. In the tame case, this follows from Proposition \ref{qsswt}, Proposition \ref{dynkin}, Corollary  \ref{mixedcomp1}, Proposition \ref{mixedcomp}, and Proposition \ref{regcomp}.
\end{proof}


\end{document}